\documentclass[11pt]{amsart}

\usepackage[pagebackref=true, colorlinks=true, citecolor=blue, pdfborder={0 0 .1},breaklinks]{hyperref}

\title[Pseudomeasure solutions of the Navier-Stokes equations]
{Improved regularity and analyticity  of Cannone-Karch solutions of the  three-dimensional Navier-Stokes equations on the torus}

\author[D.M. Ambrose]{David M. Ambrose}
\address{Department of Mathematics, Drexel University, Philadelphia, PA 19104, USA}
\email{dma68@drexel.edu}

\author[M.C. Lopes Filho]{Milton C. Lopes Filho}
\address{Instituto de Matematica, Universidade Federal do Rio de Janeiro, Caixa Postal 68530,
Rio de Janeiro, RJ, 21941-909 Brazil}
\email{mlopes@im.ufrj.br}

\author[H.J. Nussenveig Lopes]{Helena J. Nussenzveig Lopes}
\address{Instituto de Matematica, Universidade Federal do Rio de Janeiro, Caixa Postal 68530,
Rio de Janeiro, RJ, 21941-909 Brazil}
\email{hlopes@im.ufrj.br}

\newtheorem{theorem}{Theorem}
\newtheorem{lemma}{Lemma}
\newtheorem{definition}{Definition}
\newtheorem{remark}{Remark}
\newtheorem{claim}{Claim}

\newcommand{\vertiii}[1]{{\left\vert\kern-0.25ex\left\vert\kern-0.25ex\left\vert #1
    \right\vert\kern-0.25ex\right\vert\kern-0.25ex\right\vert}}

\newcommand{\torus}{\mathbb{T}}
\newcommand{\real}{\mathbb{R}}

\newcommand{\cpmtwo}{\mathcal{PM}^{2}}
\newcommand{\czfour}{\mathcal{Z}^{4}}

\begin{document}

\begin{abstract}
We consider the three-dimensional Navier-Stokes equations, with initial data having second derivatives in the space of
pseudomeasures.  Solutions of this system with such data have been shown to exist previously by Cannone and Karch.
As the Navier-Stokes equations are a parabolic system, the solutions gain regularity at positive times.  We demonstrate
an improved gain of regularity at positive times as compared to that demonstrated by Cannone and Karch.  We further
demonstrate that the solutions are analytic at all positive times, with lower bounds given for the radius of analyticity.
\end{abstract}

\maketitle

\section{Introduction}

We consider global solutions of the three-dimensional incompressible Navier-Stokes equations on a periodic spatial domain,
\begin{equation}\label{NS}
\left\{
\begin{array}{ll}
\partial_{t}v+(v\cdot\nabla) v = -\nabla p + \mu\Delta v,& (t,x)\in(0,\infty)\times\mathbb{T}^{3}, \\
\mathrm{div}(v)=0,& (t,x)\in[0,\infty)\times\mathbb{T}^{3},\\
v(\cdot,0)=v_{0},& (t,x)\in\{0\}\times\mathbb{T}^{3}.
\end{array}\right.
\end{equation}
We will specifically take $\mathbb{T}^{3}$ to be $[0,2\pi]^{3}$ with periodic boundary conditions.
Cannone and Karch have proved the existence of small global solutions for this problem when the initial data is
taken to have two derivatives in the space of pseudomeasures \cite{cannoneKarch}.

The Fourier series for a function, $f,$ on domain $\mathbb{T}^{3}$ may be written
\begin{equation}\label{fourierSeries}
f(x)=\sum_{k\in\mathbb{Z}^{3}}\hat{f}(k)e^{i k \cdot x},
\end{equation}
where $\hat{f}(k)$ is the $k$-th Fourier coefficient.

The space $PM^{a}$ for $a\geq0$ is defined below.

\begin{definition} \label{PMa} Let $a\geq 0$. Then the space of pseudomeasures $PM^{a}$ is the Banach space of distributions $f \in \mathcal{D}^\prime ( \mathbb{T}^{3})$ such that the norm $\|f\|_{PM^{a}}$ is finite, where
\begin{equation} \nonumber
\|f\|_{PM^{a}}=|\hat{f}(0)|+\sup_{k\in\mathbb{Z}^{3}}|k|^{a}|\hat{f}(k)|.
\end{equation}
\end{definition}

In what follows, we will always consider distributions with mean zero, so the inclusion of $|\hat{f}(0)|$ in calculations of norms will be unnecessary.

The Cannone-Karch result is that if the initial data is small in $PM^{2},$ then the three-dimensional Navier-Stokes
equations have a solution for all time \cite{cannoneKarch}.
Furthermore, the solution gains regularity in time, so that it is in $PM^{3-\varepsilon}$
for some $\varepsilon>0,$ at all positive times.

The space $PM^{2}$ used by Cannone and Karch is a critical space: the norm of a solution in this space is unchanged
when accounting for scaling invariance.  Other authors have studied other critical spaces for the three-dimensional
Navier-Stokes equations.  The most famous and most comprehensive result is that of Koch and Tataru, proving existence
of solutions for data in $BMO^{-1}$ \cite{kochTataru}.  Another critical space is $X^{-1},$ the space of functions which are
the derivative of a function in the Wiener algebra.  Lei and Lin proved existence of solutions for the three-dimensional
Navier-Stokes equations for $X^{-1}$ data \cite{leiLin}.  As the Navier-Stokes equations are parabolic, one naturally
expects gain of regularity of solutions as compared to the data, and one in fact expects analyticity of solutions at
positive times.  Analyticity of solutions has been proven for regular data by Foias and Temam \cite{foiasTemam}
and Grujic and Kukavica \cite{grujicKukavica}.  With data in $BMO^{-1},$ analyticity of solutions at positive times
has been established by Germain, Pavlovic, and Staffilani \cite{germainPavlovicStaffilani}.  With data in $X^{-1},$
analyticity of solutions at positive times has been established by Bae \cite{baePAMS} in the case of
spatial domain $\mathbb{R}^{3}$ and by the authors in the case of spatial domain $\mathbb{T}^{3}$ \cite{ALN4}.
Note that it is straightforward to see that neither of the spaces $X^{-1}$ and $PM^{2}$ are a subset of the other,
so the present result and the result of \cite{ALN4} are complementary.

In the present work, we first establish improved parabolic gain of regularity for the Cannone-Karch solutions with
$PM^{2}$ data, demonstrating the full gain of two derivatives expected from the Laplacian term, in a time-averaged sense.
That is, we will first show that solutions are in $PM^{4}$ at almost every $t>0.$  We then furthermore establish that the
solutions are in fact analytic at every positive time.  Our arguments are similar to the arguments of our prior work
\cite{ALN4}, which followed ideas introduced in \cite{baePAMS} and \cite{baeBiswasTadmor}, but we must introduce
a new function space on space-time in which to work.

Both the work of Cannone and Karch \cite{cannoneKarch} and the present work use a two-norm approach, but with different
forms for the second (higher-regularity) norm.
Cannone and Karch demonstrate the gain of regularity at positive times by showing that, for $2<a<3,$ the quantity
\begin{equation*}
\sup_{t>0}t^{\frac{a}{2}-1}\|u\|_{PM^{a}}
\end{equation*}
is finite.  Rather than use a supremum in time with a power of $t,$ our method uses an integral in
time of a higher norm to show the gain of regularity.
This method using an integral yields for us the full expected gain of two derivatives at positive times.
We conclude that the time-integrated form used here is more robust and gives stronger results.

The plan of the paper is as follows.  In Section \ref{preliminarySection}, we state key lemmas and define our space-time
function spaces.  In Section \ref{existenceSection}, we state and prove our first main theorem, Theorem \ref{existenceTheorem}, the content of which is existence of solutions for all time
to the three-dimensional Navier-Stokes equations with sufficiently small data in $PM^{2},$ with the previously mentioned
gain of two spatial derivatives at almost every positive time.  Then in Section \ref{analyticitySection}, we state and prove
our second main theorem, Theorem \ref{analyticityTheorem}, the content of which is that the
solutions previously proved to exist are in fact analytic at every positive time.

\section{Preliminaries}\label{preliminarySection}

We use the notations $\mathbb{Z}^{3}_{*}$ to mean $\mathbb{Z}^{3}_{*}=\mathbb{Z}^{3}\setminus\{0\}$ and, for $x = (x_1,x_2,x_3)$, $|x|_{\infty} = max\{|x_1|,|x_2|,|x_3|\}$. Note that $(\sqrt{3}/3) |x| \leq |x|_{\infty} \leq  |x|$.

We have the following key lemma.
\begin{lemma} \label{keyLemma}
There exists $c>0$ such that for all $k\in\mathbb{Z}^{3}_{*},$
\begin{equation}\nonumber
\sum_{j\in\mathbb{Z}^{3}_{*}, j\neq k}\frac{1}{|j|^{2}|k-j|^{2}}\leq\frac{c}{|k|}.
\end{equation}
\end{lemma}

\begin{proof}
Fix $k \in \mathbb{Z}^3_{*}$. We divide $\mathbb{Z}^3_{*} \setminus \{k\}$ into four regions:
\[Q_1 = \{j \in \mathbb{Z}^3_{*} \setminus \{k\}: |j|_{\infty} > 4|k|_{\infty}\}.\]
\[Q_2 = \{j \in \mathbb{Z}^3_{*} \setminus \{k\}: |j|_{\infty} < (1/4)|k|_{\infty}.\}\]
\[Q_3 = \{j \in \mathbb{Z}^3_{*} \setminus \{k\}: |k-j|_{\infty} < (1/4)|k|_{\infty}.\}\]
\[Q_4 = \{j \in \mathbb{Z}^3_{*} \setminus \{k\}: |j|_{\infty} \leq 4|k|_{\infty}\} \setminus (Q_2 \cup Q_3).\]
Next, let
\[\sum_{j\in\mathbb{Z}^{3}_{*}, j\neq k}\frac{1}{|j|^{2}|k-j|^{2}} = \sum_{i=1}^4 \sum_{j \in Q_i} \frac{1}{|j|^{2}|k-j|^{2}}\]
We will prove the $\mathcal{O}(1/|k|)$ bound for each $i=1,\ldots,4$ separately.
We first look at $i=2$. In $Q_2$,
\[|k-j| \geq \sqrt{3} |k-j|_{\infty} \geq \sqrt{3} (|k|_{\infty} - |j|_{\infty}) \geq 3\sqrt{3}/4 |k|_{\infty} \geq 3/4 |k|.\]
Therefore,
\[\sum_{j \in Q_2} \frac{1}{|j|^{2}|k-j|^{2}} \leq \frac{16}{9|k|^2} \sum_{Q_2} \frac{1}{j^2}.\]
Let $S_l = \{j \in Q_2: |j| = l\}$. We decompose $Q_2$ into the cubic shells $S_l$, so that
\[Q_2 = \bigcup_{l=1}^{[(1/4)|k|_{\infty}]} S_l,\]
where $[\ast]$ denotes the usual biggest integer less than $\ast$.
The cubic shell $S_l$, which has edges of size $2l$, has exactly $24 l^2$ integer points. Moreover, in each of these points, $1/|j|^2 \leq 1/l^2$
Therefore,
\[\sum_{j \in Q_2} \frac{1}{|j|^{2}|k-j|^{2}} \leq \frac{16}{9|k|^2} \sum_{l=1}^{[(1/4)|k|_{\infty}]}(24 l^2) \frac{1}{l^2}= C\frac{1}{|k|^2} |k|_{\infty} \leq \frac{C}{|k|},\] as desired.

Exchanging $j$ and $k$ we see that \[\sum_{j \in Q_2} \frac{1}{|j|^{2}|k-j|^{2}} = \sum_{j \in Q_3} \frac{1}{|j|^{2}|k-j|^{2}},\] which takes care of the case $i=3$ by repeating the argument above.

Next we look at $i=4$. The number of integer points in $Q_4$ is bounded by $C |k|^3$, since it is contained in a cube of edge $4 |k|_{\infty}$ and $|k|_{\infty} \leq |k|$. Each point in $Q_4$ is at a distance bigger than $C|k|$ both from the origin and from $k$. Therefore, for each $j \in Q_4$,
\[\frac{1}{|j|^{2}|k-j|^{2}} \leq \frac{C}{|k|^4},\] which implies the desired estimate.

Finally, we look at $i=1$. We decompose $Q_1$ into cubic shells, writing $Q_1 = \bigcup_{l=4|k|_{\infty}+1}^{\infty} S_l$. For each $j \in S_l$,
$|j| \geq l$, so that $1/|j|^2 \leq 1/l^2$. Also, $4|k| \leq l$ and therefore $|k-j| \geq |j| - |k| \geq  l - l/4 = Cl.$
Thus, $1/|k-j|^2 \leq C/l^2$ in $S_l$. So, again using the fact that there are less than $C|l|^2$ integer points in $S_l$, it follows that
\[ \sum_{j \in Q_1} \frac{1}{|j|^{2}|k-j|^{2}} \leq C \sum_{l=4|k|_{\infty}+1}^{\infty} \frac{1}{l^2}.\]
Next note that, as the integrand $1/x^2$ is decreasing,
\[\frac{1}{4|k|} \geq \frac{1}{4|k|_{\infty}} = \int_{4|k|_{\infty}}^{\infty} \frac{dx}{x^2} = \sum_{l=4|k|_{\infty}}^{\infty} \int_{l}^{l+1} \frac{dx}{x^2} \geq \sum_{l=4|k|_{\infty}}^{\infty} \frac{1}{(l+1)^2} = \sum_{l=4|k|_{\infty}+1}^{\infty} \frac{1}{l^2},\]
which concludes the proof.

\end{proof}

We will use the following abstract fixed point result; the authors have used this result previously in \cite{ALN6}, \cite{ALN5},
\cite{ALN4}.
\begin{lemma} \label{fixedpoint}
Let ($X,$ $\vertiii{\cdot}_{X}$) be a Banach space. Assume that $\mathcal{B}:X \times X \to X$ is a continuous bilinear operator and let $\eta>0$ satisfy $\eta\geq \|\mathcal{B}\|_{X\times X\rightarrow X}$. Then, for any $x_0 \in X$ such that
\[4\eta \vertiii{x_0}_{X}<1,\]
there exists one and only one solution to the equation
\[x=x_0+\mathcal{B}(x,x) \qquad \text{ with } \vertiii{x}_{X} < \frac{1}{2\eta}.\]
Moreover, $\vertiii{x}_{X} \leq 2\vertiii{x_0}_{X}$.
\end{lemma}
We do not provide proof here, but we instead refer the reader to \cite[p. 37, Lemma 1.2.6]{Cannone1995} and \cite{AuscherTchamitchian1999}, \cite{Cannone2003}.

We now give our function spaces.  All distributions we will consider in the sequel will have zero mean, so we define
our spaces only for distributions with zero mean.  All of our spaces will have norms defined in terms of the
Fourier coefficients as given in \eqref{fourierSeries}.
We have already defined the Banach space for our initial data, $PM^{2},$ through Definition \ref{PMa} with $a=2$.
We now need space-time versions of the pseudomeasure spaces.  The first of these will be
$\mathcal{PM}^{b}$.
\begin{definition} \label{calPMb}
Let $b \geq 0$. The space-time space of pseudomeasures is the set of distributions $f\in \mathcal{D}^\prime (\real_+\times\torus^3)$ such that the $\mathcal{PM}^b$-norm below is finite:
\begin{equation*}
\|f\|_{\mathcal{PM}^{b}}=\sup_{k\in\mathbb{Z}^{3}_{*}}\sup_{t\geq0}|k|^{b}|\hat{f}(t,k)|.
\end{equation*}
\end{definition}

A related function space on space-time will be used to demonstrate the parabolic gain of regularity.

\begin{definition} \label{calZc}
Let $c\geq 0$. We introduce the space of distributions $\mathcal{Z}^{c}$ such that the norm below is finite:
\begin{equation*}
\|f\|_{\mathcal{Z}^{c}}=\sup_{k\in\mathbb{Z}^{3}_{*}}\int_{0}^{\infty}|k|^{c}|\hat{f}(t,k)|\ \mathrm{d}t.
\end{equation*}
\end{definition}

In this work we will consider $\mathcal{PM}^{b}$ with $b=2$ and $\mathcal{Z}^{c}$ with $c=4$.

We will be applying Lemma \ref{fixedpoint} to show existence of a mild solution of the Navier-Stokes equations \eqref{NS} with data in $PM^{2}.$
We will use $X\equiv \mathcal{PM}^{2}\cap\mathcal{Z}^{4}$ for the function space  in Lemma \ref{fixedpoint}.
To this end we introduce a norm on $X=\mathcal{PM}^{2}\cap\mathcal{Z}^{4}$ given by
\begin{equation*} 
\vertiii{f}=\|f\|_{\mathcal{PM}^{2}}+\|f\|_{\mathcal{Z}^{4}}.
\end{equation*}

As mentioned above, in this work we will deal with $\mathcal{PM}^b$ only in the case $b=2$ and $\mathcal{Z}^c$ with $c=4$.
However other values of $b$ and $c$ can of course be relevant for different problems, and the authors plan to make additional use of these spaces, in particular for the Constantin-Lax-Majda equation in \cite{ALN6}.

\section{Existence of solutions}\label{existenceSection}

We now state our first main theorem.
\begin{theorem} \label{existenceTheorem}
There exists $\varepsilon>0$ such that for any $v_{0}\in PM^{2}$ satisfying
$\|v_{0}\|_{ PM^{2}} \leq \varepsilon,$ there exists a global mild solution
$v\in\mathcal{PM}^{2}\cap\mathcal{Z}^{4}$ to the initial-value problem for the incompressible Navier-Stokes equations \eqref{NS} with initial data $v_0$.
\end{theorem}

\begin{remark}
Note that, since in particular $v \in \czfour$, it follows that $v(t,\cdot) \in PM^4$ for almost every time $t\geq 0$, as announced in the introduction.
\end{remark}

The remainder of this section is devoted to the proof of this theorem.
We formulate the Navier-Stokes problem in the fixed-point form of Lemma \ref{fixedpoint}.  To begin, we let $\mathbb{P}$ be
the Leray projector, $\mathbb{P}=\mathbb{I}-\nabla\Delta^{-1}\mathrm{div}.$  Then using vector identities,
the evolution equation in \eqref{NS} becomes
\begin{equation}\nonumber
\partial_{t}v+\mathbb{P}(\mathrm{div}(v\otimes v))=\mu\Delta v.
\end{equation}
The Duhamel formula for the Navier-Stokes equations can then be written as
\begin{equation}\label{Duhamel}
v(t,\cdot)=e^{\mu t\Delta}[v_{0}]-\int_{0}^{t}e^{\mu(t-s)\Delta}[\mathbb{P}\mathrm{div}(v\otimes v)(s,\cdot)]\ \mathrm{d}s.
\end{equation}
The associated bilinear form is then
\begin{equation}\label{bilinform}
B(F,G)=-\int_{0}^{t}e^{\mu(t-s)\Delta}[\mathbb{P}\mathrm{div}(F\otimes G)(s,\cdot)]\ \mathrm{d}s.
\end{equation}
The Fourier coefficients of the nonlinear part $B(F,G)$ are, consequently,
\begin{multline}\label{BHat}
\widehat{B(F,G)}(t,k)=-\int_{0}^{t}e^{-\mu(t-s)|k|^{2}}\left[\sum_{j=1}^3 k_j \widehat{FG_j}(s,k)  \right. \\ 
\left. -\sum_{i,j=1}^{3} \frac{k}{|k|^{2}}k_{i} k_{j}
\widehat{F_{i}G_{j}}(s,k) \right]\ \mathrm{d}s
\\
=-\int_{0}^{t}e^{-\mu(t-s)|k|^{2}}\left[\sum_{j=1}^3 k_j \sum_{\ell\in\mathbb{Z}^{3}_{*}}\hat{F} (s,\ell)\hat{G}_j(s,k-\ell)  \right. \\ \left. -\sum_{i,j=1}^{3}\frac{k}{|k|^{2}}k_{i} k_{j}
\sum_{\ell\in\mathbb{Z}^{3}_{*}}\hat{F}_{i}(s,\ell)\hat{G}_j(s,k-\ell)\right]\ \mathrm{d}s.
\end{multline}

\subsection{Linear estimates}

We begin by showing that the semigroup maps $PM^{2}$ to $\mathcal{PM}^{2}\cap\mathcal{Z}^{4}.$
We start by noting that
\begin{equation}\nonumber
\|e^{\mu t \Delta} v_{0}\|_{\mathcal{PM}^{2}} =
\sup_{k\in\mathbb{Z}^{3}_{*}}\sup_{t\geq0}|k|^{2}e^{-\mu t|k|^{2}}|\hat{v}_{0}(k)|
\leq \sup_{k\in\mathbb{Z}^{3}_{*}}|k|^{2}|\hat{v}_{0}(k)|=\|v_{0}\|_{PM^{2}}.
\end{equation}
We next compute
\begin{equation}\nonumber
\|e^{\mu t\Delta}v_{0}\|_{\mathcal{Z}^{4}} =
\sup_{k\in\mathbb{Z}^{3}_{*}}\int_{0}^{\infty}|k|^{4}e^{-\mu t|k|^{2}}|\hat{v}_{0}(k)|\ \mathrm{d}t.
\end{equation}
Evaluating this integral, we have
\begin{equation}\nonumber
\|e^{\mu t\Delta}v_{0}\|_{\mathcal{Z}^{4}}=\frac{1}{\mu}\sup_{k\in\mathbb{Z}^{3}_{*}}|k|^{2}|\hat{v}_{0}(k)|
=\frac{\|v_{0}\|_{PM^{2}}}{\mu}.
\end{equation}

We have shown that
\begin{equation}\label{linpartexist}
   \vertiii{e^{\mu t\Delta}v_{0}} \leq c\|v_{0}\|_{PM^{2}},
 \end{equation}
for $c=\max\{1,\mu^{-1}\}$.

\subsection{Bilinear estimates}
Next we will show that $B$ is a continuous bilinear form on $X \times X$, for $X=\cpmtwo \cap \czfour$.

Hereafter constants are denoted by $c$ and they may vary from one line to the next.

We estimate $B(F,G)$ in the space $\mathcal{PM}^{2};$ the relevant norm may be expressed as
\begin{equation*}
\|B(F,G)\|_{\mathcal{PM}^{2}}=\sup_{k\in\mathbb{Z}^{3}_{*}}\sup_{t\geq0}|k|^{2}|\widehat{B(F,G)}(t,k)|.
\end{equation*}
Taking the absolute value of \eqref{BHat} and making some elementary manipulations (such as applying the
triangle inequality and substituting $t$ for $+\infty$), we may bound this as
\begin{multline}\nonumber
\|B(F,G)\|_{\mathcal{PM}^{2}}
\\
\leq
2\sup_{k\in\mathbb{Z}^{3}_{*}}\int_{0}^{\infty}|k|^{3}e^{-\mu(t-s)|k|^{2}}\sum_{i,j=1}^{3}
\sum_{\ell\in\mathbb{Z}^{3}_{*},\ell\neq k}|\hat{F}_{i}(s,\ell)| |\hat{G}_{j}(s,k-\ell)|\ \mathrm{d}s.
\end{multline}
We neglect the exponential and rearrange slightly, arriving at
\begin{equation}\label{commonForm}
\|B(F,G)\|_{\mathcal{PM}^{2}}\leq2\sup_{k\in\mathbb{Z}^{3}_{*}}\sum_{i,j=1}^{3}\int_{0}^{\infty}|k|^{3}
\sum_{\ell\in\mathbb{Z}^{3}_{*},\ell\neq k}|\hat{F}_{i}(s,\ell)| |\hat{G}_{j}(s,k-\ell)|\ \mathrm{d}s.
\end{equation}
We next bound $|k|^{2}$ by $2(|k-\ell|^{2}+|\ell|^{2}),$ finding
\begin{multline}\nonumber
\|B(F,G)\|_{\mathcal{PM}^{2}}\leq 4\sup_{k\in\mathbb{Z}^{3}_{*}}\sum_{i,j=1}^{3}\int_{0}^{\infty}
|k|\sum_{\ell\in\mathbb{Z}^{3}_{*},\ell\neq k} |\ell|^{2}|\hat{F}(s,\ell)||\hat{G}(s,k-\ell)|\ \mathrm{d}s
\\
+4\sup_{k\in\mathbb{Z}^{3}_{*}}\sum_{i,j=1}^{3}\int_{0}^{\infty}
|k|\sum_{\ell\in\mathbb{Z}^{3}_{*},\ell\neq k} |\hat{F}(s,\ell)| |k-\ell|^{2}|\hat{G}(s,k-\ell)|\ \mathrm{d}s.
\end{multline}
We next adjust factors of $k,$ $\ell,$ and $k-\ell:$
\begin{multline}\nonumber
\|B(F,G)\|_{\mathcal{PM}^{2}}
\\
\leq 4\sup_{k\in\mathbb{Z}^{3}_{*}}\sum_{i,j=1}^{3}\int_{0}^{\infty}
|k|\sum_{\ell\in\mathbb{Z}^{3}_{*},\ell\neq k}\frac{1}{|k-\ell|^{2}|\ell|^{2}}
|\ell|^{4}|\hat{F}_{i}(s,\ell)||k-\ell|^{2}|\hat{G}_{j}(s,k-\ell)|\ \mathrm{d}s
\\
+4\sup_{k\in\mathbb{Z}^{3}_{*}}\sum_{i,j=1}^{3}\int_{0}^{\infty}
|k|\sum_{\ell\in\mathbb{Z}^{3}_{*},\ell\neq k}\frac{1}{|k-\ell|^{2}|\ell|^{2}}
|\ell|^{2} |\hat{F}_{i}(s,\ell)| |k-\ell|^{4}|\hat{G}_{j}(s,k-\ell)|\ \mathrm{d}s.
\end{multline}
We now take the appropriate supremums so as to be able to
recognize norms of $F$ and $G,$ arriving at the penultimate bound
\begin{multline*}
\|B(F,G)\|_{\mathcal{PM}^{2}}
\\
\leq 4
\left(\sup_{k\in\mathbb{Z}^{3}_{*}}|k|\sum_{\ell\in\mathbb{Z}^{3}_{*},\ell\neq k}\frac{1}{|k-\ell|^{2}|\ell|^{2}}\right)
\left(\|F\|_{\mathcal{Z}^{4}}\|G\|_{\mathcal{PM}^{2}}+\|F\|_{\mathcal{PM}^{2}}\|G\|_{\mathcal{Z}^{4}}\right).
\end{multline*}
By Lemma \ref{keyLemma}, we have our final estimate for $B(F,G)$ in $\mathcal{PM}^{2},$ namely
\begin{equation}\label{nonlincpmtwoexist}
\|B(F,G)\|_{\mathcal{PM}^{2}}\leq c\vertiii{F}\vertiii{G}.
\end{equation}

\begin{remark} \label{rhsimportantbd}
Observe that what we have effectively shown is that the expression appearing on the right-hand-side of \eqref{commonForm} may be estimated by $c\vertiii{F}\vertiii{G}$. This expression will appear several times throughout this work.
\end{remark}

Next we must estimate $B(F,G)$ in $\mathcal{Z}^{4};$ we may express the corresponding norm as
\begin{equation}\nonumber
\|B(F,G)\|_{\mathcal{Z}^{4}}=\sup_{k\in\mathbb{Z}^{3}_{*}}|k|^{4}\int_{0}^{\infty}|\widehat{B(F,G)}(t,k)|\ \mathrm{d}t.
\end{equation}
Substituting from \eqref{BHat} and making elementary estimates (such as applying the triangle inequality, we find
the bound
\begin{multline}\nonumber
\|B(F,G)\|_{\mathcal{Z}^{4}}
\\
\leq 2\sup_{k\in\mathbb{Z}^{3}_{*}}\sum_{i,j=1}^{3}
\int_{0}^{\infty}\int_{0}^{t}|k|^{5}e^{-\mu(t-s)|k|^{2}}\sum_{\ell\in\mathbb{Z}^{3}_{*},\ell\neq k}
|\hat{F}_{i}(s,\ell)| |\hat{G}_{j}(s,k-\ell)| \ \mathrm{d}s\mathrm{d}t.
\end{multline}
We change the order of integration, finding
\begin{multline}\nonumber
\|B(F,G)\|_{\mathcal{Z}^{4}}
\\
\leq 2\sup_{k\in\mathbb{Z}^{3}_{*}}\sum_{i,j=1}^{3}
\int_{0}^{\infty}\int_{s}^{\infty}|k|^{5}e^{-\mu(t-s)|k|^{2}}\sum_{\ell\in\mathbb{Z}^{3}_{*},\ell\neq k}
|\hat{F}_{i}(s,\ell)| |\hat{G}_{j}(s,k-\ell)| \ \mathrm{d}t\mathrm{d}s.
\end{multline}
We evaluate the integral with respect to $t,$ finding
\begin{equation}\nonumber
\|B(F,G)\|_{\mathcal{Z}^{4}}\leq \frac{2}{\mu}\sup_{k\in\mathbb{Z}^{3}_{*}}\sum_{i,j=1}^{3}
\int_{0}^{\infty}|k|^{3}e^{-\mu s|k|^{2}}\sum_{\ell\in\mathbb{Z}^{3}_{*},\ell\neq k}
|\hat{F}_{i}(s,\ell)| |\hat{G}_{j}(s,k-\ell)| \ \mathrm{d}s.
\end{equation}
Neglecting this remaining exponential, we have
\begin{equation}\nonumber
\|B(F,G)\|_{\mathcal{Z}^{4}}\leq \frac{2}{\mu}\sup_{k\in\mathbb{Z}^{3}_{*}}\sum_{i,j=1}^{3}
\int_{0}^{\infty}|k|^{3}\sum_{\ell\in\mathbb{Z}^{3}_{*},\ell\neq k}
|\hat{F}_{i}(s,\ell)| |\hat{G}_{j}(s,k-\ell)| \ \mathrm{d}s.
\end{equation}
Except for a constant factor, this right-hand side is the same as the one in \eqref{commonForm}.
We recall Remark \ref{rhsimportantbd} to conclude that
\begin{equation}\label{nonlinczfourexist}
\|B(F,G)\|_{\mathcal{Z}^{4}}\leq c\vertiii{F}\vertiii{G}.
\end{equation}

It follows from \eqref{nonlincpmtwoexist} and \eqref{nonlinczfourexist} that
\begin{equation} \label{nonlinpartexist}
\vertiii{B(F,G)} \leq c\vertiii{F}\vertiii{G},
\end{equation}
thereby establishing the desired continuity of the bilinear form on $X$.

The proof of Theorem \ref{existenceTheorem} follows in a standard way from
the linear and bilinear estimates \eqref{linpartexist} and \eqref{nonlinpartexist} together with Lemma \ref{fixedpoint}.

\section{Analyticity of solutions}\label{analyticitySection}

In this section we show that the solutions proven to exist in Theorem \ref{existenceTheorem} are actually analytic at positive times, provided the initial data is sufficiently small.

\begin{theorem}\label{analyticityTheorem}
Let $\alpha \in (0,1)$. There exists $\varepsilon>0$ such that for any $v_{0}\in PM^{2}$ satisfying
$\|v_{0}\|_{ PM^{2}} \leq \varepsilon,$ the solution
$v\in\mathcal{PM}^{2}\cap\mathcal{Z}^{4}$ obtained in Theorem \ref{existenceTheorem} with initial data $v_0$ is analytic, with radius of analyticity $R_\alpha \geq \max\{\mu \sqrt{t},\mu\alpha t\}$.
\end{theorem}

The rest of this section is dedicated to the proof of Theorem \ref{analyticityTheorem}.

For $b(t) \geq 0$ such that $b(0)=0$ we consider the operator $e^{\mu b(t) |D|}$, where $|D|$ is the operator whose Fourier multiplier is $|k|$. In what follows we will be interested in two cases:

\begin{enumerate}
  \item $b(t) = \sqrt{t}$ and
  \item $b(t) = \alpha t$.
\end{enumerate}

Recall the Duhamel formula for the solution $v$ given in \eqref{Duhamel}. Let
$V=V(t,\cdot) \equiv e^{\mu b(t) |D|} [v(t,\cdot)]$. Then, since $v_0 = V_0$, we have:
\begin{multline} \label{Duhamelweighted}
V(t,\cdot)=
\\
e^{\mu b(t) |D|} e^{ \mu\Delta} V_{0} -\int_{0}^{t}e^{\mu b(t) |D|} e^{\mu(t-s)\Delta}[\mathbb{P}\mathrm{div}(e^{-\mu b(s)|D|}V \otimes e^{-\mu b(s)|D|}V)(s,\cdot)]
\mathrm{d}s.
\end{multline}

The new bilinear form we must deal with is

\begin{equation}\label{bilinform weighted}
B(F,G)=\int_{0}^{t}e^{\mu b(t) |D|} e^{\mu(t-s)\Delta}[\mathbb{P}\mathrm{div}(e^{-\mu b(s)|D|}F \otimes e^{-\mu b(s)|D|}G)(s,\cdot)]
\mathrm{d}s.
\end{equation}

We compute the Fourier coefficients of this new bilinear form:
The Fourier coefficients of the nonlinear part $B(F,G)$ are, consequently,
\begin{multline}\label{BweightedHat}
\widehat{B(F,G)}(t,k)=
\\
-\int_{0}^{t}e^{\mu b(t) |k|-\mu(t-s)|k|^{2}} \left[ \sum_{j=1}^3 k_j \widehat{e^{-\mu b(s)|D|}[F]e^{-\mu b(s)|D|}[G_{j}]}(s,k) \right. \\
\left. -\sum_{i,j=1}^{3}\frac{kk_{i}}{|k|^{2}} k_{j}
\widehat{e^{-\mu b(s)|D|}[F_{i}]e^{-\mu b(s)|D|}[G_{j}]}(s,k) \right] \ \mathrm{d}s.
\end{multline}

After expressing the Fourier coefficients of the products above as discrete convolutions it follows that
\begin{multline}\label{absvalBweightedHat}
|\widehat{B(F,G)}(t,k)|
\\
\leq 2|k|
\int_{0}^{t}e^{\mu b(t) |k|-\mu b(s) |k | -\mu(t-s)|k|^{2}}\sum_{i,j=1}^{3}
\sum_{\ell\in\mathbb{Z}^{3}_{*}} |\hat{F}_{i}(s,\ell)|
|\hat{G}_j(s,k-\ell)|\ \mathrm{d}s.
\end{multline}

\subsection{Linear estimates for analyticity}

As with the existence theorem we begin by showing that the operator $e^{\mu b(t) |D|} e^{ \mu\Delta} $ maps $PM^2$ continuously into $X = \cpmtwo \cap \czfour$ in both cases $b(t) = \sqrt{t}$ and $b(t)=\alpha t$.

\begin{claim}
\begin{itemize}
\item[]

\item[(i)]
If $b(t) = \sqrt{t}$ then $e^{\mu b(t)|k|-\mu t |k|^2}\leq e^{\mu/2}e^{-\mu t |k|^2/2}$.

\item[(ii)]
If $b(t) = \alpha t$ then $e^{\mu b(t)|k|-\mu t |k|^2}\leq e^{-(1-\alpha)\mu t |k|^2}$.
\end{itemize}
\end{claim}

The proof of Claim 1 (i) is contained in the proof of estimate \cite[(20)]{ALN4} and the proof of Claim 1 (ii) is part of the proof of estimate \cite[(23)]{ALN4}.

Fix $v_0 \equiv V_0 \in PM^2$. We first estimate $e^{\mu b(t) |D|} e^{ \mu\Delta} V_0 $ in $\cpmtwo$.

Let $b(t) = \sqrt{t}$. Then we use Claim 1(i) to find
\begin{align*}
\|e^{\mu b(t) |D|} e^{ \mu\Delta} V_0 \|_{\cpmtwo} & = \sup_{k\in \mathbb{Z}^3_{\ast}}\sup_{ t\geq 0} |k|^2 e^{\mu b(t) |k|-\mu t |k|^2} |\hat{V_0}(k)| \\
&\leq \sup_{k\in \mathbb{Z}^3_{\ast}}\sup_{ t\geq 0} e^{\mu/2} e^{-\mu t|k|^2/2}|k|^2 |\hat{V_0}(k)| \leq e^{\mu/2}\|V_0\|_{PM^2}.
\end{align*}
Next let $b(t) = \alpha t$. Then we can use Claim 1 (ii) to get
\begin{align*}
\|e^{\mu b(t) |D|} e^{ \mu\Delta} V_0 \|_{\cpmtwo} & = \sup_{k\in \mathbb{Z}^3_{\ast}}\sup_{ t\geq 0} |k|^2 e^{\mu b(t) |k|-\mu t |k|^2} |\hat{V_0}(k)| \\
&\leq   \sup_{k\in \mathbb{Z}^3_{\ast}}\sup_{ t\geq 0} e^{-\mu (1-\alpha) t|k|^2 }|k|^2 |\hat{V_0}(k)| \leq  \|V_0\|_{PM^2}.
\end{align*}

Now we will estimate $e^{\mu b(t) |D|} e^{ \mu\Delta} V_0 $ in $\czfour$.

Consider first $b(t)=\sqrt{t}$. Then we again use Claim 1 (i) to obtain
\begin{align*}
\|e^{\mu b(t) |D|} e^{ \mu\Delta} V_0 \|_{\czfour} & = \sup_{k\in \mathbb{Z}^3_{\ast}} \int_0^{+\infty}|k|^4 e^{\mu b(t) |k|-\mu t |k|^2} |\hat{V_0}(k)| \,\mathrm{d} s \\
&\leq \sup_{k\in \mathbb{Z}^3_{\ast}}\int_0^{+\infty}|k|^4 e^{\mu/2} e^{-\mu t|k|^2/2} |\hat{V_0}(k)| \,\mathrm{d} s\\
&= \sup_{k\in \mathbb{Z}^3_{\ast}}|k|^4 e^{\mu/2} \frac{2}{\mu |k|^2} |\hat{V_0}(k)| \\
&=  \frac{2e^{\mu/2}}{\mu}\|V_0\|_{PM^2}.
\end{align*}

Finally, assume that $b(t) = \alpha t$. Then, using Claim 1 (ii), we have
\begin{align*}
\|e^{\mu b(t) |D|} e^{ \mu\Delta} V_0 \|_{\czfour} & = \sup_{k\in \mathbb{Z}^3_{\ast}} \int_0^{+\infty}|k|^4 e^{\mu b(t) |k|-\mu t |k|^2} |\hat{V_0}(k)| \,\mathrm{d} s \\
&\leq \sup_{k\in \mathbb{Z}^3_{\ast}}\int_0^{+\infty}|k|^4 e^{-\mu (1-\alpha)t|k|^2} |\hat{V_0}(k)| \,\mathrm{d} s\\
&= \sup_{k\in \mathbb{Z}^3_{\ast}}|k|^4 \frac{1}{\mu(1-\alpha) |k|^2} |\hat{V_0}(k)| \\
&=  \frac{1}{\mu (1-\alpha)}\|V_0\|_{PM^2}.
\end{align*}

We have shown that
\begin{equation} \label{linestanalyticity}
\vertiii{e^{\mu b(t) |D|} e^{ \mu\Delta} V_0} \leq c\|V_0\|_{PM^2} \equiv c\|v_0\|_{PM^2},
\end{equation}
with $c=\max\left\{e^{\mu/2},1,\frac{2e^{\mu/2}}{\mu},\frac{1}{\mu (1-\alpha)}\right\}$.

\subsection{Bilinear estimates for analyticity}

Now it remains to show that the new bilinear form \eqref{bilinform weighted} is continuous on $X \times X$.

As before, we begin by estimating $B(F,G)$ in $\cpmtwo$. We will use throughout the estimate on the absolute value of $\widehat{B(F,G)}(t,k)$ in \eqref{absvalBweightedHat}. We obtain
\begin{multline} \label{Bestcpmtwoanalyticity}
\|B(F,G)\|_{\cpmtwo}  = \sup_{k\in \mathbb{Z}^3_{\ast}}\sup_{ t\geq 0} |k|^2 |\widehat{B(F,G)}|(t,k)
\leq
\\
\sup_{k\in \mathbb{Z}^3_{\ast}}\sup_{ t\geq 0} 2|k|^3
\int_{0}^{t}e^{\mu b(t) |k|-\mu b(s) |k | -\mu(t-s)|k|^{2}}\sum_{i,j=1}^{3}
\sum_{\ell\in\mathbb{Z}^{3}_{*}} |\hat{F}_{i}(s,\ell)|
|\hat{G}_j(s,k-\ell)|\ \mathrm{d}s.
\end{multline}

We have two weights $b(t)$ to discuss, namely $b(t) = \sqrt{t}$ and $b(t) = \alpha t$.

\begin{claim}
\begin{itemize}
\item[]

\item[(i)]
If $b(t) = \sqrt{t}$ then, for all $0 \leq s \leq t$, $e^{\mu b(t) |k|-\mu b(s) |k | -\mu(t-s)|k|^{2}}\leq
e^{\mu/2}e^{-\mu (t-s) |k|^2/2}$.

\item[(ii)]
If $b(t) = \alpha t$ then, for all $0 \leq s \leq t$, $e^{\mu b(t) |k|-\mu b(s) |k | -\mu(t-s)|k|^{2}}
\leq e^{-(1-\alpha)\mu (t-s) |k|^2}$.
\end{itemize}
\end{claim}

The proof of Claim 2 (i) is contained in the proof of estimate \cite[(21)]{ALN4} and the proof of Claim 2 (ii) is part of the proof of the estimate on the bilinear term $\widehat{\overline{B}(F,G)}(t,k)$ following estimate \cite[(25)]{ALN4}.

In particular, it follows from Claim 2 that, in both cases $b(t) = \sqrt{t}$ and $b(t) = \alpha t$,
\begin{equation} \label{expestforcpmtwo}
e^{\mu b(t) |k|-\mu b(s) |k | -\mu(t-s)|k|^{2}}
\leq c,
\end{equation}
for some constant $c>0$.

Inserting \eqref{expestforcpmtwo} into \eqref{Bestcpmtwoanalyticity} we deduce that
\begin{align} \label{Bestcpmtwoanalyticitynew}
\|B(F,G)\|_{\cpmtwo} & \leq \sup_{k\in \mathbb{Z}^3_{\ast}}\sup_{ t\geq 0} 2|k|^3
\int_{0}^{t} c \sum_{i,j=1}^{3}
\sum_{\ell\in\mathbb{Z}^{3}_{*}} |\hat{F}_{i}(s,\ell)|
|\hat{G}_j(s,k-\ell)|\ \mathrm{d}s. \nonumber \\
&\leq 2c \sup_{k\in \mathbb{Z}^3_{\ast}}\sum_{i,j=1}^{3}
\int_{0}^{+\infty}|k|^3
\sum_{\ell\in\mathbb{Z}^{3}_{*}} |\hat{F}_{i}(s,\ell)|
|\hat{G}_j(s,k-\ell)|\ \mathrm{d}s.
\end{align}
The right-hand-side of \eqref{Bestcpmtwoanalyticitynew} is, up to a constant, the same as the right-hand-side of \eqref{commonForm}. We thus recall Remark \ref{rhsimportantbd} to find
\begin{equation} \label{Bnewcpmtwo}
\|B(F,G)\|_{\cpmtwo} \leq c \vertiii{F}\vertiii{G}.
\end{equation}.

Next we estimate $B(F,G)$ in $\czfour$. We have, using \eqref{absvalBweightedHat},
\begin{multline} \label{Bestczfouranalyticity}
\|B(F,G)\|_{\czfour}  = \sup_{k\in \mathbb{Z}^3_{\ast}}\int_0^{+\infty} |k|^4 |\widehat{B(F,G)}|(t,k)\,\mathrm{d}t
\\
\leq \sup_{k\in \mathbb{Z}^3_{\ast}}\int_0^{+\infty} 2|k|^5
\int_{0}^{t}\Bigg[e^{\mu b(t) |k|-\mu b(s) |k | -\mu(t-s)|k|^{2}}
\\
\sum_{i,j=1}^{3}
\sum_{\ell\in\mathbb{Z}^{3}_{*}} |\hat{F}_{i}(s,\ell)|
|\hat{G}_j(s,k-\ell)|\Bigg] \mathrm{d}s\mathrm{d}t.
\end{multline}

First assume $b(t) = \sqrt{t}$. Then, from Claim 2 (i), we find
\begin{multline} \label{Bestczfouranalyticitysqrtt}
\|B(F,G)\|_{\czfour} \leq
\\
 \sup_{k\in \mathbb{Z}^3_{\ast}}\int_0^{+\infty} 2|k|^5
\int_{0}^{t}e^{\mu/2}e^{-\mu(t-s)|k|^2/2}\sum_{i,j=1}^{3}
\sum_{\ell\in\mathbb{Z}^{3}_{*}} |\hat{F}_{i}(s,\ell)|
|\hat{G}_j(s,k-\ell)|\ \mathrm{d}s\mathrm{d}t
\\
= \sup_{k\in \mathbb{Z}^3_{\ast}}\int_0^{+\infty} 2|k|^5
\int_{s}^{+\infty}e^{\mu/2}e^{-\mu(t-s)|k|^2/2}\sum_{i,j=1}^{3}
\sum_{\ell\in\mathbb{Z}^{3}_{*}} |\hat{F}_{i}(s,\ell)|
|\hat{G}_j(s,k-\ell)|\ \mathrm{d}t\mathrm{d}s
\\
= \sup_{k\in \mathbb{Z}^3_{\ast}}\int_0^{+\infty} 2|k|^5
 e^{\mu/2} \frac{2}{\mu |k|^2} \sum_{i,j=1}^{3}
\sum_{\ell\in\mathbb{Z}^{3}_{*}} |\hat{F}_{i}(s,\ell)|
|\hat{G}_j(s,k-\ell)|\ \mathrm{d}s
 \\
= \frac{4}{\mu}e^{\mu/2}  \sup_{k\in \mathbb{Z}^3_{\ast}}\sum_{i,j=1}^{3}
\int_{0}^{+\infty}|k|^3
\sum_{\ell\in\mathbb{Z}^{3}_{*}} |\hat{F}_{i}(s,\ell)|
|\hat{G}_j(s,k-\ell)|\ \mathrm{d}s.
\end{multline}

Next suppose $b(t) = \alpha t$. Then, from Claim 2 (ii) it follows that
\begin{multline} \label{Bestczfouranalyticityalphat}
\|B(F,G)\|_{\czfour} \leq
\\
\sup_{k\in \mathbb{Z}^3_{\ast}}\int_0^{+\infty} 2|k|^5
\int_{0}^{t}e^{-(1-\alpha)\mu (t-s) |k|^2}\sum_{i,j=1}^{3}
\sum_{\ell\in\mathbb{Z}^{3}_{*}} |\hat{F}_{i}(s,\ell)|
|\hat{G}_j(s,k-\ell)|\ \mathrm{d}s\mathrm{d}t
\\
= \sup_{k\in \mathbb{Z}^3_{\ast}}\int_0^{+\infty} 2|k|^5
\int_{s}^{+\infty}e^{-(1-\alpha)\mu (t-s) |k|^2}\sum_{i,j=1}^{3}
\sum_{\ell\in\mathbb{Z}^{3}_{*}} |\hat{F}_{i}(s,\ell)|
|\hat{G}_j(s,k-\ell)|\ \mathrm{d}t\mathrm{d}s
\\
= \sup_{k\in \mathbb{Z}^3_{\ast}}\int_0^{+\infty} 2|k|^5
  \frac{1}{(1-\alpha)\mu |k|^2} \sum_{i,j=1}^{3}
\sum_{\ell\in\mathbb{Z}^{3}_{*}} |\hat{F}_{i}(s,\ell)|
|\hat{G}_j(s,k-\ell)|\ \mathrm{d}s
 \\
= \frac{2}{ (1-\alpha)\mu}  \sup_{k\in \mathbb{Z}^3_{\ast}}\sum_{i,j=1}^{3}
\int_{0}^{+\infty}|k|^3
\sum_{\ell\in\mathbb{Z}^{3}_{*}} |\hat{F}_{i}(s,\ell)|
|\hat{G}_j(s,k-\ell)|\ \mathrm{d}s.
\end{multline}

We note that the right-hand-side of both \eqref{Bestczfouranalyticitysqrtt} and \eqref{Bestczfouranalyticityalphat} are, up to a constant, the same as the right-hand-side of \eqref{commonForm}. As before we recall Remark \ref{rhsimportantbd} to deduce that
\begin{equation} \label{Bnewczfour}
\|B(F,G)\|_{\czfour} \leq c \vertiii{F}\vertiii{G}.
\end{equation}.

Putting together \eqref{Bnewcpmtwo} and \eqref{Bnewczfour} we arrive at
\begin{equation} \label{nonlinpartanalyticity}
\vertiii{B(F,G)} \leq c\vertiii{F}\vertiii{G},
\end{equation}
which shows that the weighted in time bilinear form is continuous on $X \times X$.

We can now use Lemma \ref{fixedpoint} in a standard way to produce a solution $V = e^{\mu b(t)|D|}[v(t,\cdot)] \in X$ satisfying the Duhamel formula \eqref{Duhamelweighted} for both $b(t) = \sqrt{t}$ and $b(t) = \alpha t$.

Since $V$ is in $\mathcal{PM}^{2},$ we immediately see that the Fourier coefficients of $v$ must decay, with exponential decay rate $\mu b(t).$
The exponential decay of a function's Fourier coefficients at a given rate implies analyticity of the function in a strip about the real axis, with radius of analyticity at
least as big as the exponential decay rate, by the periodic analogue of Theorem IX.13 of \cite{reedSimon}.
We therefore conclude the analyticity of $v(t,\cdot)$ for any $t > 0$, with radius of analyticity bounded below by
$\max\{\mu \sqrt{t},\mu \alpha t\}$, thereby concluding the proof of Theorem \ref{analyticityTheorem}.

\section*{Acknowledgments} DMA gratefully acknowledges support from the National Science Foundation through grant
DMS-2307638. MCLF was partially supported by CNPq, through grant \# 304990/2022-1, and FAPERJ, through  grant \# E-26/201.209/2021.
HJNL acknowledges the support of CNPq, through  grant \# 305309/2022-6, and of FAPERJ, through  grant \# E-26/201.027/2022.

\bibliographystyle{plain}
\bibliography{improved-regularity-ns-pm.bib}

\end{document}